\documentclass{amsart}
\usepackage{amssymb}
\usepackage{graphicx}
\usepackage[shortlabels]{enumitem}
\usepackage{multirow}
\usepackage{amsmath,color}
\usepackage{hyperref}
\usepackage{url}
\usepackage{longtable}
\usepackage{tikz}

\theoremstyle{plain}
\newtheorem{theorem}{Theorem}[section]

\newtheorem{proposition}[theorem]{Proposition}
\newtheorem{lemma}[theorem]{Lemma}
\newtheorem{corollary}[theorem]{Corollary}

\theoremstyle{definition}

\theoremstyle{remark}
\newtheorem{remark}[theorem]{Remark}

{

\newcommand{\Z}{\mathbb{Z}}

\begin{document}
\title[Equiangular lines and regular graphs]{Equiangular line systems and \\ switching classes containing regular graphs}

\author[G. R. W. Greaves]{ Gary R. W. Greaves }


\address{School of Physical and Mathematical Sciences, 
Nanyang Technological University, 21 Nanyang Link, Singapore 637371}
\email{grwgrvs@gmail.com}
\thanks{The author was supported by JSPS KAKENHI; 
grant number: 26$\cdot$03903 and by MOE AcRF; grant number: 2016-T1-002-067-01}

	

\begin{abstract}
	We develop the theory of equiangular lines in Euclidean spaces.
	Our focus is on the question of when a Seidel matrix having precisely three distinct eigenvalues has a regular graph in its switching class.
	We make some progress towards an answer to this question by finding some necessary conditions and some sufficient conditions.
	Furthermore, we show that the cardinality of an equiangular line system in $18$ dimensional Euclidean space is at most $60$.
\end{abstract}

\maketitle

\section{Introduction and preliminaries}

Let $d$ be a positive integer and let $\mathcal L = \{ l_1, \dots , l_n \}$ be a system of lines in $\mathbb R^d$, where each $l_i$ is spanned by a unit vector $\mathbf{v}_i$.
The line system $\mathcal L$ is called \textbf{equiangular} if there exists a constant $\alpha>0$ such that the inner product $\langle \mathbf{v}_i, \mathbf{v}_j \rangle = \pm \alpha$ for all $i \ne j$.
(The constant $\alpha$ is called the \emph{common angle}.)
Let $N(d)$ denote the maximum cardinality of a system of equiangular lines in dimension $d$.
Determining values of the sequence $\{N(d)\}_{d \in \mathbb N}$ is a classical problem that has received much attention~\cite{lemmens1973equiangular,SeidelVanLint66, Seidel1974:EulerSwitchOdd} and recently \cite{barg2014new,bukh2016bounds,glazyrin2016upper, Greaves2016208, Szollosi2017kq} there have been some improvements to the upper bounds for $N(d)$ for various values of $d$.
One contribution of this article is to improve the upper bound for $N(18)$ showing that $N(18) \leqslant 60$.
Furthermore, we show that certain Seidel matrices corresponding to systems of $60$ equiangular lines in $\mathbb R^{18}$ each must contain in their switching classes a regular graph having four distinct eigenvalues (see Remark~\ref{rem:607290graph} below).

In Table~\ref{tab:equi} below, we give the currently known (including the improvement from this paper) values or lower and upper bounds for $N(d)$ for $d$ at most $23$.

\begin{table}[htbp]
	\begin{center}
		\setlength{\tabcolsep}{2pt}
	\begin{tabular}{c|cccccccccccccccc}
		$d$  & 2 & 3        & 4           & 5  & 6  & 7--13 & 14  & 15 & 16 & 17 & 18 & 19 & 20 & 21 & 22 & 23\\\hline
		$N(d)$  & 3 & 6        & 6           & 10 & 16 & 28    & 28--29  & 36 & 40--41 & 48--50 & 54--60 & 72--75 & 90--95  & 126 & 176 & 276
	\end{tabular}
	\end{center}
	\caption{Bounds for the sequence $N(d)$ for $2\leqslant d\leqslant 23$.}
	\label{tab:equi}
\end{table}

For tables of bounds for equiangular line systems in larger dimensions we refer the reader to Barg and Yu~\cite{barg2014new}.

A \textbf{Seidel matrix} is a symmetric $\{0,\pm 1\}$-matrix $S$ with zero diagonal and all off-diagonal entries nonzero.
It is well-known~\cite{SeidelVanLint66} that Seidel matrices and equiangular line systems (with positive common angle) are equivalent: for $i\ne j$, the inner product $\langle \mathbf{v}_i, \mathbf{v}_j \rangle = a$ (with $a \in \{-\alpha,\alpha\}$) if and only if the $(i,j)$th entry of its corresponding Seidel matrix is $a/\alpha$.
Moreover, if $\mathcal L$ is an equiangular line system of $n$ lines in $\mathbb R^d$ with common angle $\alpha > 0$ then its corresponding Seidel matrix $S$ has smallest eigenvalue $-1/\alpha$ with multiplicity $n-d$.
Conversely, given a Seidel matrix $S$ with smallest eigenvalue $-1/\alpha <0$ whose multiplicity is $n-d$, we can construct an equiangular line system of $n$ lines in $\mathbb R^d$ with common angle $\alpha$.

Let $\mathcal O_n(\mathbb{Z})$ denote the orthogonal group generated by signed permutation matrices of order $n$ and let $X$ and $Y$ be two $\mathbb Z$-matrices of order $n$.
We say that $X$ and $Y$ are \textbf{switching equivalent} (denoted $X \cong Y$) if $X = P^{\top}YP$ for some matrix $P \in \mathcal O_n(\mathbb{Z})$.
It is clear that two switching-equivalent matrices have the same spectrum.
We will sometimes use the verb ``to switch'' to mean conjugation by diagonal matrices from $\mathcal O_n(\mathbb{Z})$.

The symbols $I$, $J$, and $O$ will (respectively) always denote the identity matrix, the all-ones matrix, and the all-zeros matrix; the order of each matrix should be clear from the context in which it is used, however, the order will sometimes be indicated by a subscript.
We use $\mathbf{1}$ to denote the all-ones (column) vector.

Let $S$ be a Seidel matrix.
Then $A = (J-I-S)/2$ is the adjacency matrix for a graph $\Gamma$, which we call the \textbf{underlying graph} of $S$.
The set of underlying graphs of Seidel matrices that are switching equivalent to $S$ is called the \textbf{switching class} of $S$.
Throughout this paper we use techniques from spectral graph theory and we refer to \cite{BrouwerHaemers:Spectra} for the necessary background.

We will be concerned with the question of when a Seidel matrix has a regular graph in its switching class.
Using regular graphs to construct systems of equiangular lines is not new.
Indeed, De Caen~\cite{de2000large} used a family of previously studied regular graphs to give a construction of an infinite family of large systems of equiangular lines.
Also it was shown by Taylor~\cite{Taylor77} that any principal submatrix of order $n-1$ of an $n \times n$ Seidel matrix $S$ having precisely two distinct eigenvalues must have a strongly regular graph in its switching class.

In this paper we give some necessary and some sufficient conditions for when a Seidel matrix with precisely three distinct eigenvalues has a regular graph in its switching class.
The above result of Taylor relating Seidel matrices with precisely two distinct eigenvalues to strongly regular graphs is a special case of our results.
Indeed, a principal submatrix of order $n-1$ of an $n \times n$ Seidel matrix with precisely two distinct eigenvalues has precisely three distinct eigenvalues~\cite[Lemma 5.12]{Greaves2016208}.
So the above result of Taylor says that certain Seidel matrices having precisely three distinct eigenvalues have a (strongly) regular graph in their switching classes.
In fact this result of Taylor is a special case of Corollary~\ref{cor:simple} below.

The paper is organised as follows.
In Section~\ref{sec:seidel_matrices_with_few_eigenvalues} we present some basic results about Seidel matrices.
In Section~\ref{sec:regular_graphs_in_the_switching_class} we motivate the question of whether Seidel matrices with precisely three distinct eigenvalues have a regular graph in their switching classes and we find some necessary conditions.
We prove some lemmas in Section~\ref{sec:positive_semidefinite_matrices} for use in Section~\ref{sec:regular_eigenspaces_of_seidel_matrices_with_three_eigenvalues} where we prove our sufficient conditions for when a Seidel matrix with three eigenvalues has a regular graph in its switching class.
In Section~\ref{sec:relative_bound} we extend the work in \cite{Greaves2016208} of strengthening the relative bound and we give some open problems in Section~\ref{sec:open_problems}.

\section{Basic properties of Seidel matrices} 
\label{sec:seidel_matrices_with_few_eigenvalues}

%
%

\subsection{Eigenvalues of Seidel matrices} 
\label{sub:eigenvalues_of_seidel_matrices}

Let $M$ be a real symmetric matrix with $r$ distinct eigenvalues $\theta_1 < \dots < \theta_r$ such that $\theta_i$ has multiplicity $m_i$.
We write the spectrum of $M$ as $\{ [\theta_1]^{m_1},\dots,[\theta_r]^{m_r} \}$.
Less accurately, we may write that a real symmetric matrix has spectrum $\{\{ [\theta_1]^{m_1},\dots,[\theta_r]^{m_r} \}\}$ implying that $\theta_i$ may be equal to $\theta_j$ for some $i$ and $j$.
In this instance the multiplicity of an eigenvalue $\theta$ of $M$ is equal to the sum of the $m_i$ for which $\theta_i = \theta$. 

Let $S$ be a Seidel matrix of order $n$.
Then $S$ is a real symmetric matrix with diagonal entries $0$. 
Moreover, the diagonal entries of $S^2$ are all equal to $n-1$.
Putting the above facts about Seidel matrices together, we obtain the following equations for the traces of $S$ and $S^2$:
\begin{align}
	\operatorname{tr} S &= 0; \label{eqn:tr} \\
	\operatorname{tr} S^2 &= n(n-1). \label{eqn:tr2}
\end{align}

%


\subsection{Seidel matrices over finite rings} 
\label{sub:seidel_matrices_over_finite_rings}

Let $\mathcal M_n$ denote the ring of integer matrices of order $n$.
Let $S$ be a Seidel matrix.
Since we can write $S = J - I - 2A$ where $A$ is a graph adjacency matrix, the next lemma follows immediately.

\begin{lemma}\label{lem:mod2}
	Let $S$ be a Seidel matrix of order $n$.
	Then modulo $2\mathcal M_n$ we have
	\[
	S^k \equiv
		\begin{cases}
			J-I, &  \text{ if $k$ is odd;} \\
			nJ-I, & \text{ if $k$ is even.}
		\end{cases}
	\]
\end{lemma}

Next we have the following lemma about matrices congruent to $J-I$ modulo $2\mathcal M_n$.\footnote{This result, due to Willem Haemers, was pointed out to the author by Jack Koolen.}
We denote the characteristic polynomial of a matrix $M$ by $\chi_M(x) := \det(xI - M)$.

\begin{lemma}\label{lem:Neumann}
	Let $M$ be an $n\times n$ matrix congruent to $J-I$ modulo $2\mathcal M_n$.
	Then modulo $2\mathbb Z[x]$ we have
	\[
		\chi_S(x) \equiv
		\begin{cases}
			(x+1)^n & \text{ if $n$ is even, } \\
			x(x+1)^{n-1} & \text{ if $n$ is odd.}
		\end{cases}
	\]
\end{lemma}

It follows from this lemma that, if its order is even then, a Seidel matrix cannot have any even integer eigenvalues.
Furthermore, a Seidel matrix of odd order must have a simple eigenvalue.
Indeed, we record this consequence as a corollary.
 
\begin{corollary}\label{cor:oddSeidel}
	Let $S$ be a Seidel matrix of odd order.
	Then $S$ has an eigenvalue of multiplicity $1$.
\end{corollary}
\begin{proof}
	Suppose the characteristic polynomial of $S$ factorises as $f_1(x)^{m_1}\dots f_r(x)^{m_r}$, where each $f_i(x)$ is an irreducible monic integer polynomial.
	Since $\mathbb Z[x] / 2\mathbb Z[x]$ is a unique factorisation domain, one of the $f_i(x)$, ($f_1(x)$ say) must be congruent to $x(x+1)^{l}$ modulo $2\mathbb Z[x]$ for some $l \in \mathbb N$.
	It is then easy to see that $m_1$ must be equal to $1$.
\end{proof}

An \textbf{Euler graph} is a graph all of whose vertices have even degree.
The next lemma gives sufficient conditions for when a Seidel matrix has an Euler graph in its switching class.

\begin{lemma}[See {\cite[Theorem 5.2]{Greaves2016208}}, {\cite[Theorem 1]{Hage2003:Switching}}, and \cite{Seidel1974:EulerSwitchOdd}]\label{lem:Eulergraph}
	Let $S$ be a Seidel matrix of order $n$ having precisely $r$ distinct eigenvalues.
	Suppose that $n$ is odd or $r = 3$.
	Then $S$ is switching equivalent to the matrix $J-I-2A$ where $A$ is the adjacency matrix of an Euler graph.
\end{lemma}

We can strengthen Lemma~\ref{lem:mod2} for Seidel matrices whose underlying graphs are Euler graphs.
For example, we have the following lemma.
\begin{lemma}\label{lem:ScubedEuler}
	Let $S$ be an $n \times n$ Seidel matrix whose underlying graph $\Gamma$ is an Euler graph.
	Then
	\[
		S^2 \equiv (n-2)J + I \mod 4\mathcal M_n
	\]
\end{lemma}
\begin{proof}
	Let $A$ be the adjacency matrix of $\Gamma$.
	Since $\Gamma$ is an Euler graph, we have that $2A J \equiv O \mod 4\mathcal M_n$.  
	Hence $SJ = (J-I-2A)J \equiv (n-1)J \mod 4\mathcal M_n$.
	The lemma then follows in a straightforward manner.
\end{proof}

Given a Seidel matrix with precisely three distinct eigenvalues $\lambda$, $\mu$, and $\nu$, define the matrix $M_S(\lambda,\mu)$ as
\[
	M_S(\lambda,\mu) := \sigma (S-\lambda I)(S-\mu I),
\]
where $\sigma := \operatorname{sgn}\rho$ for $\rho = (\nu-\lambda)(\nu-\mu)$.
Note that $M_S(\lambda,\mu)$ is always positive semidefinite and the diagonal entries of $M_S(\lambda,\mu)$ are each equal to $|n-1+\lambda\mu|$.
Further observe that, since $M_S(\lambda,\mu)$ is positive semidefinite, its off-diagonal entries each must have absolute value at most $|n-1+\lambda\mu|$.

\begin{lemma}\label{lem:MformOdd}
	Let $S$ be a Seidel matrix of odd order $n$ having precisely three distinct eigenvalues.
	Then $S$ has two distinct eigenvalues $\lambda$ and $\mu$ satisfying
	\[
		M_S(\lambda,\mu) \cong |n-1+\lambda\mu|J.
	\]
\end{lemma}
\begin{proof}
	By Corollary~\ref{cor:oddSeidel}, the matrix $S$ must have at least one eigenvalue $\nu$ (say) having multiplicity $1$.
	If $\lambda$ and $\mu$ are the other two eigenvalues of $S$ then $M_S(\lambda,\mu)$ is a symmetric, rank-$1$ matrix.
	Since the diagonal entries of $M_S(\lambda,\mu)$ are equal to $|n-1+\lambda\mu|$, we have that $M_S(\lambda,\mu)$ is switching equivalent to $|n-1+\lambda\mu|J$, as required.
\end{proof}

\begin{lemma}\label{lem:MformEven}
	Let $S$ be a Seidel matrix of even order $n$ having precisely three distinct eigenvalues.
	Then $S$ has two distinct eigenvalues $\lambda$ and $\mu$ satisfying the congruence
	\[
		(x-\lambda)(x-\mu) \equiv (x+1)^2 \mod 2\mathbb Z[x].
	\]
	Furthermore, $M_S(\lambda,\mu) \cong M$ for some $M \equiv |n-1+\lambda \mu|J \mod 4 \mathcal M_n$.
\end{lemma}
\begin{proof}
	If all three eigenvalues are rational then the first congruence follows from Lemma~\ref{lem:Neumann} for any two of the eigenvalues.
	Otherwise, assume that $S$ has at least one irrational eigenvalue.
	By \cite[Corollary 5.5]{Greaves2016208}, the matrix $S$ must have at least one rational eigenvalue $\nu$, say.
	Hence the other two eigenvalues, $\lambda$ and $\mu$, must be quadratic algebraic integers.
	Since $\mathbb Z[x]/2\mathbb Z[x]$ is a unique factorisation domain, it follows from Lemma~\ref{lem:Neumann} that $(x-\lambda)(x-\mu) \equiv (x+1)^2 \mod 2\mathbb Z[x]$.
	Thus we have shown the first congruence.
	
	By Lemma~\ref{lem:Eulergraph}, we can assume that the underlying graph of $S$ is an Euler graph.
	Then using the first congruence and Lemma~\ref{lem:ScubedEuler} we have
	\begin{align*}
		M_S(\lambda,\mu) &= \pm (S^2-(\lambda+\mu)S +\lambda\mu I) \\
		&\equiv (n-2)J-(\lambda+\mu)(J-I) +(\lambda\mu+1) I \mod 4\mathcal M_n \\
		&= (n-2-\lambda-\mu)J + (\lambda+1)(\mu+1)I.
	\end{align*}
	Hence, each off-diagonal entry of $M_S(\lambda,\mu)$ is congruent to $\pm(n-2-\lambda-\mu)$ modulo $4$.
	Moreover we have $M_S(\lambda,\mu) \equiv O \mod 2 \mathcal M_n$.
	On the other hand, each diagonal entry of $M_S(\lambda,\mu)$ is equal to $|n-1+\lambda \mu|$.
	It therefore suffices to show that $(n-2-\lambda-\mu)$ is congruent to $|n-1+\lambda \mu|$ modulo $4$ or, equivalently, that $(n-2-\lambda-\mu)/2 \equiv |n-1+\lambda \mu|/2 \mod 2$.
	
	Suppose to the contrary that $(n-2-\lambda-\mu)/2$ and $|n-1+\lambda \mu|/2$ have different parities.
	Then, modulo $2 \mathcal M_n$, the matrix $M_S(\lambda,\mu)/2$ is congruent to either $J-I$ or $I$.
	If $M_S(\lambda,\mu)/2$ is congruent to $J-I$ modulo $2 \mathcal M_n$ then, by Lemma~\ref{lem:Neumann}, the characteristic polynomial of $M_S(\lambda,\mu)/2$ is congruent to $(x+1)^n$ modulo $2\Z[x]$.
	Similarly, if the matrix $M_S(\lambda,\mu)/2$ is congruent to $I$ modulo $2 \mathcal M_n$ then its characteristic polynomial is again congruent to $(x+1)^n$ modulo $2\Z[x]$.
	Therefore, in either case, the matrix $M_S(\lambda,\mu)/2$ cannot have any even eigenvalues, which is a contradiction since $M_S(\lambda,\mu)/2$ does not have full rank.
\end{proof}

As a corollary to the proof of Lemma~\ref{lem:MformEven} we have the following result.

\begin{corollary}\label{cor:Mform}
	Let $S$ be a Seidel matrix of order $n$ even with spectrum $\{ [\lambda]^a, [\mu]^b, [\nu]^c \}$ such that $\lambda\mu \in \mathbb Z$.
	Then
	\[
		M_S(\lambda,\mu) \cong M \text{ for some } M \equiv |n-1+\lambda \mu|J \mod 4 \mathcal M_n.
	\]
\end{corollary}

\section{Regular graphs in the switching class of a Seidel matrix with three eigenvalues} 
\label{sec:regular_graphs_in_the_switching_class}

A \textbf{regular} graph is a graph all of whose vertices have valency $k$ for some $k \in \mathbb N$.
In this section we give some necessary (spectral) conditions for Seidel matrices to contain regular graphs in their switching classes.
We are particularly interested in the case when a Seidel matrix with precisely three distinct eigenvalues has a regular graph in its switching class.

Seidel matrices having precisely three distinct eigenvalues were first systematically studied by Greaves et al.~\cite{Greaves2016208}.
The authors collated various interesting - because of their relation to large systems of equiangular lines - constructions of Seidel matrices having three eigenvalues (see \cite[Section 5.2]{Greaves2016208}).
We observe that these constructions all give rise to Seidel matrices each of whose switching classes contain a regular graph.
Moreover, below we show that any Seidel matrix having the spectrum of \cite[Example 5.10]{Greaves2016208} or \cite[Example 5.19]{Greaves2016208} must necessarily have a regular graph in its switching class.
It remains an open question as to whether there exists a Seidel matrix cospectral with \cite[Example 5.17]{Greaves2016208} or \cite[Example 5.18]{Greaves2016208} whose switching class does not contain a regular graph (see Section~\ref{sec:open_problems}).

However, we know of at least one example of a Seidel matrix having precisely three distinct eigenvalues whose switching class does not contain a regular graph.
Indeed, the $10 \times 10$ Seidel matrix
\[ \mathfrak S =
	\begin{pmatrix}
	    0 & 1 & 1  & 1  & 1  &1  & 1  & 1 & 1 & 1 \\
	    1 & 0 & -1 &  1 & 1  &-1 & 1  & -1& 1 & -1 \\
	    1 & -1 &  0 &  1 & 1  &-1 & -1 & 1 & -1& 1 \\
	    1 & 1 & 1  & 0  & 1  &-1 & -1 & 1 & -1& -1 \\
	    1 & 1 & 1  & 1  & 0  &-1 & 1  & 1 & 1 & -1 \\
	    1 & -1 &  -1 & -1 & -1 & 0 & 1  & 1 & 1 & 1 \\
	    1 & 1 & -1 &  -1&  1 & 1 & 0  & -1& 1 & -1 \\
	    1 & -1 &  1 &  1 &  1 & 1 & -1 & 0 & 1 & 1 \\
	    1 & 1 & -1 &  -1&  1 & 1 & 1  & 1 & 0 & 1 \\
	    1 & -1 &  1 &  -1&  -1& 1 & -1 & 1 & 1 & 0 
	\end{pmatrix}
\]
has spectrum $\{ [-3]^4, [2+\sqrt{5}]^3, [2-\sqrt{5}]^3 \}$ and its switching class does not contain a regular graph.
All Seidel matrices up to order $12$ have been computed (see \cite[Section 4]{Greaves2016208}).
Using this computation we find that, except for the above Seidel matrix $\mathfrak S$, every Seidel matrix of order at most $12$ with precisely three distinct eigenvalues has a regular graph in its switching class.
We therefore ask the following question:
\paragraph{\textbf{Question A}} 
\label{par:questionA}
Does every Seidel matrix with precisely three distinct \emph{rational} eigenvalues contain a regular graph in its switching class?

We can also ask if $\mathfrak S$ is the only exceptional Seidel matrix amongst those having three distinct eigenvalues:
\paragraph{\textbf{Question B}} 
\label{par:questionB}
Do there exist any Seidel matrices of order at least $14$ with precisely three distinct eigenvalues whose switching class does not contain a regular graph?

In Section~\ref{sec:strongly_regular_graphs_and_regular_two_graphs}, we will see that all Seidel matrices of odd order having precisely three distinct eigenvalues contain a regular graph in their switching classes.
Hence in Question B we only ask about Seidel matrices of order at least $14$.

\begin{remark}
	\label{rem:OSquestion}
	Question B has already been answered by \"Osterg\aa rd and Sz\"oll\H{o}si~\cite[Theorem 4.8]{Ostergard17}, who exhibited examples of Seidel matrices of order $18$ having precisely three distinct eigenvalues and no regular graph in their switching classes.
	The examples are of the form $S \cong \pm (J_3 \otimes (S_6-I_6) + I_{18})$ where $S_6$ is a Seidel matrix having spectrum $\{ [-\sqrt{5}]^3, [\sqrt{5}] \}$.
	\"Osterg\aa rd and Sz\"oll\H{o}si~\cite{Ostergard17} further suggested the matrices $(J_{2k+1} \otimes (S_6-I_6) + I_{6(2k+1)})$ for $k \geqslant 1$ may be good candidates for Seidel matrices that do not have a regular graph in their switching classes.
	Below, we show that this is indeed the case (see Remark~\ref{rem:Ostergard}).
\end{remark}

Questions A and B motivate the remainder of this section and the subsequent two sections.

If a Seidel matrix $S$ has a regular graph $\Gamma$ in its switching class then there is a simple correspondence between the eigenvalues of $S$ and the eigenvalues of $\Gamma$.
(Here the eigenvalues of a graph are taken to be the eigenvalues of its adjacency matrix.)
Indeed, since the adjacency matrix of $\Gamma$ and $J-I$ commute, we have the following lemma.

\begin{lemma}\label{lem:regSei}
	Let $\Gamma$ be a connected regular graph with valency $k$ and spectrum $\{ [k]^1, [\lambda_1]^{m_1}, \dots, [\lambda_r]^{m_r} \}$.
	Then its Seidel matrix has spectrum $\{ \{ [n-1-2k]^1, [-1-2\lambda_1]^{m_1}, \dots, [-1-2\lambda_r]^{m_r} \}\}$.
\end{lemma}

Next we give a condition on the diagonal entries of $S^3$ where $S$ is a Seidel matrix having a regular graph in its switching graph.

\begin{lemma}\label{lem:regSwitch}
	Let $S = J - I - 2A$ where $A$ is the adjacency matrix of an $n$-vertex $k$-regular graph $\Gamma$.
	Then, for all $v \in V(\Gamma)$, the number of closed $3$-walks from $v$ is 
	\begin{equation*}
		\frac{(n-1)(n-2)-6k(n-2k) - [S^3]_{v,v}}{16}.
	\end{equation*}
	Moreover, $k = (n-1-\theta)/2$ for some $\theta \in \operatorname{spec} S$.
\end{lemma}
\begin{proof}
	Firstly, since $\Gamma$ is regular, the matrices $A$ and $S$ commute.
	Hence $k = (n-1-\theta)/2$ where $\theta$ is an eigenvalue of $S$.
	
	Since $S = J - I - 2A$ and $AJ = JA = kJ$, we have
	\[
		S^3 = (n^2-3n+3)J - I - 6k(n - 2k)J + 12kJ - 6A - 12A^2 - 8A^3.
	\]
	And, since the diagonal entries of $A^2$ equal $k$, we have
	\[
		[S^3]_{v,v} = (n-1)(n-2)-6k(n-2k) - 8[A^3]_{v,v}.
	\]
	The result follows because $[A^3]_{v,v}$ equals $2$ times the number of closed $3$-walks from $v$.
\end{proof}

Now let $S$ be a Seidel matrix of order $n$ having precisely three distinct eigenvalues $\lambda$, $\mu$, and $\nu$.
Then
\[
	S^3 -(\lambda+\mu+\nu)S^2 + (\lambda\mu + \lambda\nu + \mu\nu)S - \lambda\mu\nu I = (S-\lambda I)(S-\mu I)(S-\nu I) = O.
\]
By considering the diagonal entries we find that, for all $i \in \{1,\dots,n\}$, we have $[S^3]_{i,i} = (\lambda+\mu+\nu)(n-1)+\lambda\mu\nu$.
Hence the following corollary is immediate.

\begin{corollary}\label{cor:regSwitch3}
	Let $S$ be a Seidel matrix of order $n$ with precisely three distinct eigenvalues $\lambda$, $\mu$, and $\nu$.
	Suppose that $S$ has a $k$-regular graph $\Gamma$ in its switching class.
	Then, for all $v \in V(\Gamma)$, the number of closed $3$-walks from $v$ is
	\begin{equation}
		\label{eqn:s3diag}
		\frac{(n-1)(n-2)-6k(n-2k) - (\lambda+\mu+\nu)(n-1)-\lambda\mu\nu}{16}.
	\end{equation}
	Moreover, $k = (n-1-\theta)/2$ for some $\theta \in \{\lambda,\mu,\nu\}$.
\end{corollary}

Since \eqref{eqn:s3diag} must be a nonnegative integer, Corollary~\ref{cor:regSwitch3} can be used to show that certain Seidel matrices do not have a regular graph in their switching classes.
In particular, Corollary~\ref{cor:regSwitch3} can be used to verify that the Seidel matrix $\mathfrak S$ does not have a regular graph in its switching class.
Furthermore we make the following remarks.

\begin{remark}
	\label{rem:noreg40}
	By Corollary~\ref{cor:regSwitch3}, since \eqref{eqn:s3diag} must be an integer, a Seidel matrix with spectrum $\{ [-5]^{24}, [5]^6, [9]^{10} \}$ cannot have a regular graph in its switching class.
	See Table~\ref{tab:largeSets}.
\end{remark}

\begin{remark}
	\label{rem:Ostergard}
	Let $S(k) = (J_{2k+1} \otimes (S_6-I_6) + I_{6(2k+1)})$ for some $k \geqslant 1$ where $S_6$ is a Seidel matrix with spectrum $\{ [-\sqrt{5}]^3, [\sqrt{5}]^3 \}$.
	Then $S(k)$ has spectrum $\{ [-\sqrt{5}(2k+1)-2k]^3, [1]^{12k}, [\sqrt{5}(2k+1)-2k]^3 \}$.
	By Corollary~\ref{cor:regSwitch3}, since \eqref{eqn:s3diag} must be an integer, we see that, for all $k \geqslant 1$, the Seidel matrix $S(k)$ does not have a regular graph in its switching class.
	Whence we obtain an infinite family of Seidel matrices that have precisely three distinct eigenvalues and no regular graph in their switching classes.
	See Remark~\ref{rem:OSquestion}.
\end{remark}


\section{Positive semidefinite matrices with constrained entries} 
\label{sec:positive_semidefinite_matrices}

In this section we concern ourselves with positive semidefinite matrices with constant diagonals and constrained entries.
The results proved here will be applied in the subsequent section.

For a matrix $M$ and row-index $i$, define the set $V_{M,i}(k) := \{ j : M_{i,j} = k \}$.

\begin{lemma}\label{lem:repDiag}
	Let $M$ be a positive semidefinite $\mathbb R$-matrix of order $n$ with constant diagonal $d > 0$ and let $\mathbf{r}_k$ and $\mathbf{c}_k$ denote the $k$-th row and column of $M$ respectively.
	Then, for each $i \in \{1,\dots,n\}$ and each $j \in V_{M,i}(\pm d)$, we have $\mathbf{r}_j = \pm \mathbf{r}_i$ and $\mathbf{c}_j = \pm \mathbf{c}_i$.
\end{lemma}
\begin{proof}
	Since $M$ has constant diagonal $d$, for all $k \in \{1,\dots,n\}$, the set $V_{M,k}(d) \cup V_{M,k}(-d)$ contains at least one element, $k$ itself.
	Consider the $i$-th row, $\mathbf{r}_i$, of $M$.
	Without loss of generality we can switch the rows and columns of $M$ so that $\mathbf{r}_i$ does not have any entries equal to $-d$.
	Let $M^\prime$ denote the matrix obtained after performing this switching.
	If $|V_{M^\prime,i}(d)| = 1$ then there is nothing to prove.
	Suppose that $|V_{M^\prime,i}(d)| \geqslant 2$.
	
	Take $j \in V_{M^\prime,i}(d)\backslash \{i\}$ and $k \in \{1,\dots,n\} \backslash \{i,j\}$, let $A$ be a $3 \times 3$ principal submatrix of $M$ induced on the $i$-th, $j$-th, and $k$-th rows (and corresponding columns) of $M$.
	Then $A$ is positive semidefinite having the form
	\[
		\begin{pmatrix}
			d & d & a \\
			d & d & b \\
			a & b & d
		\end{pmatrix}.
	\]
	Since we have $\det A = d^3 +2dab - d^3 -d(a^2 + b^2) \geqslant 0$ we see that $a = b$.
	Hence $\mathbf{r}_i = \mathbf{r}_j$ and, since $M^\prime$ is symmetric, we also have $\mathbf{c}_i = \mathbf{c}_j$.
\end{proof}

\begin{corollary}\label{lem:repDiagTensor}
	Let $M$ be a positive semidefinite $\mathbb R$-matrix of order $n$ with constant diagonal $d > 0$.
	Suppose that each row of $M$ has precisely $k$ entries with absolute value $d$.
	Then $n=qk$ for some $q$ and $M$ is switching equivalent to the matrix $N \otimes J_k$, where $N$ is an $\mathbb R$-matrix of order $q$ with constant diagonal $d$ whose off-diagonal entries have absolute value less than $d$.
\end{corollary}
\begin{proof}
	By switching and permuting the rows and columns of $M$ we can, without loss of generality, assume that the first row of $M$ does not contain any entries equal to $-d$ and the first $k$ entries are equal to $d$.
	By Lemma~\ref{lem:repDiag}, the first $k$ rows and columns of $M$ are equal.
	
	Now we can write $n = qk + r$ with $0 \leqslant r < k$.
	For each $s \in \{ 1,\dots,q\}$ inductively apply the above argument to the $((s-1)k + 1)$-th row so that for all $\xi \in \{ (s-1)k + 1,\dots,sk \}$ we have  $V_{M,\xi}(d) \cup V_{M,\xi}(-d) = \{ (s-1)k + 1,\dots,sk \}$.
	
	For a contradiction we suppose that $r > 0$.
	Then $|V_{M,qk+r}(d) \cup V_{M,qk+r}(-d)| = k$ but $V_{M,qk+r}(d) \cup V_{M,qk+r}(-d)$ must be a subset of $\{qk+1,\dots,qk+r\}$, which has $r<k$ elements.
	Hence $r = 0$.
	
	Since $M$ is positive semidefinite with constant diagonal $d$, every principal submatrix must also be positive semidefinite.
	Hence, the off-diagonal entries have absolute value at most $d$.
\end{proof}

\begin{lemma}\label{lem:diagM}
	Let $a$ and $b$ be real numbers and let $M$ be a symmetric $\{\pm a, \pm b \}$-matrix of order $n$ with precisely two distinct eigenvalues $\lambda$ and $\mu$.
	Then 
	\[
		|\{ j : (M_{i,j})^2 = b^2 \}| =  \frac{(\lambda+\mu)M_{i,i} - na^2 - \lambda\mu}{b^2-a^2}.
	\]
\end{lemma}
\begin{proof}
	The matrix $M$ satisfies the equation $M^2 = (\lambda+\mu)M - \lambda\mu I$.
	The lemma follows since the diagonal entries of $M^2$ are equal to $s_i(b^2-a^2) + na^2$, where $s_i = |\{ j : (M_{i,j})^2 = b^2 \}|$.
\end{proof}

\begin{lemma}\label{lem:primeDiag}
	Let $M$ be a positive semidefinite $\mathbb Z$-matrix with constant diagonal $p \equiv 3 \pmod 4$ prime.
	If $M$ has rank $2$ then $M$ is switching equivalent to a matrix with precisely two distinct rows.
\end{lemma}
\begin{proof}
	Let $A$ be a $3 \times 3$ principal submatrix of $M$.
	Then $A$ is singular with the form
	\[
		\begin{pmatrix}
			p & a & b \\
			a & p & c \\
			b & c & p
		\end{pmatrix},
	\]
	where $a$, $b$, and $c$ all have absolute value at most $p$.
	Since $\det A = p^3 +2abc - p(a^2 + b^2 + c^2) = 0$, we have that $p$ divides $2abc$.
	Further, since $p$ is an odd prime, at least one of $a$, $b$, and $c$ must be divisible by $p$.
	
	For a contradiction, suppose that $a$, $b$, and $c$ all have absolute value strictly less than $p$.
	Then at least one of them $a$ (say) must be equal to $0$.
	But since $A$ has rank at most $2$, we have the equation $p^2 = b^2+c^2$.
	Hence, since $p \equiv 3 \pmod 4$, we must have $\pm p \in \{b,c\}$, which contradicts our supposition.
	
	Therefore at least one of $a$, $b$, and $c$ has absolute value equal to $p$.
	Up to switching equivalence, we can assume that $a = p$.
	Then, by Lemma~\ref{lem:repDiag}, we have that $b = c$.
	Inductively applying this argument to the rest of the matrix $M$ gives the result.
\end{proof}


Finally we prove a necessary condition for Seidel matrices having four distinct integral eigenvalues, one of which is simple.
The conclusion leads to an improvement on the maximum cardinality of an equiangular line system in $\mathbb R^{18}$ (see Remark~\ref{rem:61nonexi}).

\begin{lemma}\label{lem:posSemiPol}
	Let $S$ be a Seidel matrix of order $n$ and let $p(x) = x^3 - c_2 x^2 + c_1 x - c_0$ and assume that $\sigma p(S)$ is positive semidefinite for some $\sigma = \pm 1$.
	Then each diagonal entry of $\sigma S^3$ is at least $\sigma ((n-1)c_2 + c_0)$.
\end{lemma}
\begin{proof}
	Since $p(S)$ is positive semidefinite, each diagonal entry is at least zero and the diagonal entries of $S^2$ and $S$ are all $n-1$ and $0$ respectively.
\end{proof}

\begin{corollary}\label{cor:posSemiTr}
	Let $S$ be a Seidel matrix of odd order $n$ with integral spectrum $\{ [\theta_0]^1, [\theta_1]^{m_1}, [\theta_2]^{m_2}, [\theta_3]^{m_3} \}$ where $\theta_0$ is even.
	Set $\sigma = \operatorname{sgn}\prod_{i=1}^3 (\theta_0-\theta_i)$.
	Then 
	\[
		\sigma \theta_0^3 \geqslant n + \sigma \left( n(n-1)(\theta_1+\theta_2+\theta_3) + n\theta_1\theta_2\theta_3 - \sum_{i=1}^3 m_i \theta_i^3 \right ).
	\]
\end{corollary}
\begin{proof}
	Define $p(x) = (x-\theta_1)(x-\theta_2)(x-\theta_3)$.
	Then $\sigma p(S)$ is positive semidefinite.
	Hence, by Lemma~\ref{lem:posSemiPol}, each diagonal entry of $\sigma S^3$ is at least $\sigma ((n-1)(\theta_1+\theta_2+\theta_3) + \theta_1\theta_2\theta_3)$, which, from using Lemma~\ref{lem:Neumann}, we see is an odd integer.
	By Lemma~\ref{lem:mod2}, each diagonal entry of $\sigma S^3$ is even.
	Therefore each diagonal entry of $\sigma S^3$ is at least $\sigma ((n-1)(\theta_1+\theta_2+\theta_3) + \theta_1\theta_2\theta_3) + 1$.
	This gives a lower bound for the trace
	\[
		\operatorname{tr} \sigma S^3 \geqslant n \left (\sigma ((n-1)(\theta_1+\theta_2+\theta_3) + \theta_1\theta_2\theta_3) + 1 \right ).
	\]
	The result follows since $\operatorname{tr} \sigma S^3 = \sigma (\theta_0^3 + \sum_{i=1}^3 m_i \theta_i^3)$.
\end{proof}

\begin{remark}
	\label{rem:61nonexi}
	We remark that, by Corollary~\ref{cor:posSemiTr}, there cannot exist a Seidel matrix having spectrum $\left \{ [-5]^{43}, [11]^{9}, [12]^1, [13]^{8} \right \}$.
	We will see in Section~\ref{sec:relative_bound} that there exists a system of $61$ equiangular lines in $\mathbb R^{18}$ if and only if there exists such a Seidel matrix.
	See Remark~\ref{rem:61}.
\end{remark}

%


\section{Regular eigenspaces of Seidel matrices with three eigenvalues} 
\label{sec:regular_eigenspaces_of_seidel_matrices_with_three_eigenvalues}

Let $S$ be a Seidel matrix of order $n$ with distinct eigenvalues $\lambda_0 < \lambda_1 < \dots < \lambda_r$ and corresponding eigenspaces $E_0, E_1, \dots, E_r$.
We call an eigenspace \textbf{regular} if it contains a vector all of whose entries are $\pm 1$.
It is easy to show that the property of having a regular eigenspace is invariant under switching equivalence operations and a Seidel matrix has a regular eigenspace if and only if it has a regular graph in its switching class.
Moreover, we have the following result.

\begin{lemma}\label{lem:regular}
	Let $S$ be an $n \times n$ Seidel matrix with $r+1$ distinct eigenvalues $\theta_i$ each having multiplicity $m_i$ for $i \in \{0,\dots,r\}$.
	Suppose that the $\theta_0$-eigenspace is regular.
	Then there exists a $(\frac{n-1-\theta_0}{2})$-regular graph in the switching class of $S$ with spectrum
	\[
		\left \{ [(n-1-\theta_0)/2]^1, [(-1-\theta_0)/2]^{m_0-1}, [(-1-\theta_1)/2]^{m_1}, \dots, [(-1-\theta_r)/2]^{m_r} \right \}.
	\]
\end{lemma}
\begin{proof}
	Let $\mathbf{x}$ be a $\theta_0$-eigenvector all of whose entries are $\pm 1$.
	Set $D = \operatorname{diag}(\mathbf{x})$ and $S^\prime = DSD$.
	Then the all-ones vector, $\mathbf{1}$, is in the $\theta_0$-eigenspace $E_0$ of $S^\prime$.
	Furthermore, we can form an orthogonal basis $B_0$ for $E_0$ that contains the vector $\mathbf{1}$.
	Then $A := (J-I-S^\prime)/2$ is the adjacency matrix for the underlying graph $\Gamma$ of $S^\prime$.
	Now clearly, $\mathbf{1}$ is a $(\frac{n-1-\theta_0}{2})$-eigenvector of $A$.
	Hence each vertex of $\Gamma$ has valency $(n-1-\theta_0)/2$.
	Each $\theta_0$-eigenvector in $B_0 \backslash \{ \mathbf{1} \}$ is a $(\frac{-1-\theta_0}{2})$-eigenvector for $A$.
	And, for all $i \geqslant 1$, each $\theta_i$-eigenvector of $S^\prime$ is a $(\frac{-1-\theta_i}{2})$-eigenvector for $A$.
\end{proof}

\begin{corollary}\label{cor:regularParity}
	Let $S$ be a Seidel matrix of order $n$ and let $\nu$ be an eigenvalue of $S$ whose eigenspace is regular.
	Then $\nu$ is an integer satisfying $\nu \equiv n-1 \mod 2$ and furthermore, $\nu \equiv n-1 \mod 4$ if $n$ is odd. 
\end{corollary}
\begin{proof}
	By Lemma~\ref{lem:regular}, there exists a graph $\Gamma$ in the switching class of $S$ that is regular with valency $k = (n-1-\nu)/2$.
	Hence we have that $(n-1-\nu)/2$ is an integer and the first congruence is established.
	Furthermore, it is well-known that for an $n$-vertex, $k$-regular graph with $e$ edges the following equation is satisfied: $nk = 2e$.
	The second congruence is now clear.
\end{proof}

In Corollary~\ref{cor:regularParity} we see that we obtain more restrictions when we have a Seidel matrix of odd order.
Indeed, we now focus on Seidel matrices of odd order in the next subsection.


\subsection{Strongly regular graphs and Seidel matrices of odd order} 
\label{sec:strongly_regular_graphs_and_regular_two_graphs}

In this subsection we show a correspondence between strongly regular graphs on an odd number of vertices and Seidel matrices of odd order having precisely three distinct eigenvalues.
Here we do not give the definition of a strongly regular graph.
For us, it is enough to know that a regular graph with precisely three distinct eigenvalues (with at least one eigenvalue having multiplicity $1$) is \emph{strongly regular}.

\begin{proposition}\label{pro:regSei}
	Let $S$ be an $n \times n$ Seidel matrix with spectrum $\left \{ [\lambda]^{a}, [\mu]^{b}, [\nu]^{1} \right \}$ where $a$ and $b$ are positive integers.
	Then the $\nu$-eigenspace is regular.
\end{proposition}

\begin{proof}
	Let $M := M_S(\lambda,\mu)$.
	Since $M$ is a rank-$1$ matrix we can write $M = \mathbf{x}\mathbf{x}^\top$, where $\mathbf{x}$ is an $\nu$-eigenvector for $S$.
	Both $S^2$ and $S$ have a constant diagonal and hence $M$ also has a constant diagonal.
	Therefore the entries of $\mathbf{x}$ have constant absolute value and so the $\nu$-eigenspace is regular.
\end{proof}

Combining Proposition~\ref{pro:regSei} with Lemma~\ref{lem:regular} gives us the following result.

\begin{corollary}\label{cor:simple}
	Let $S$ be a Seidel matrix of order $n$ with spectrum $\left \{ [\lambda]^{a}, [\mu]^{b}, [\nu]^{1} \right \}$ where $a$ and $b$ are positive integers.
	Then the switching class of $S$ contains a (strongly) regular graph with spectrum
	\[
		\left \{ [(-1-\lambda)/2]^{a}, [(-1-\mu)/2]^{b}, [(n-1-\nu)/2]^{1} \right \}.
	\]
\end{corollary}

Let $S$ be a Seidel matrix of odd order having precisely three distinct eigenvalues.
By Corollary~\ref{cor:oddSeidel}, the matrix $S$ has a simple eigenvalue.
Hence, by Corollary~\ref{cor:simple}, there must exist a strongly regular graph in the switching class of $S$.
On the other hand, let $A$ be the adjacency matrix of a strongly regular graph of odd order.
Using Lemma~\ref{lem:regSei}, we see that the Seidel matrix $J-I-2A$ must have precisely three distinct eigenvalues.
Therefore we observe that Seidel matrices of odd order having precisely three distinct eigenvalues are in correspondence with strongly regular graphs with an odd number of vertices.
In particular we can remark on the nonexistence of putative Seidel matrices on $49$, $75$, and $95$ vertices. 

\begin{remark}\label{rem:no497595}
	Seidel matrices having the following spectra do not exist:
	\[
		\left \{ [-5]^{32}, [9]^{16}, [16]^1 \right \}, \; \left \{ [-5]^{56}, [10]^{1}, [15]^{18} \right \}, \text{ and } \left \{ [-5]^{75}, [14]^{1}, [19]^{19} \right \}.
	\]
	Indeed, by Corollary~\ref{cor:simple}, the existence such Seidel matrices would respectively imply the existence regular graphs having the following spectra: $\left \{ [16]^1, [2]^{32}, [-5]^{16} \right \}$, $\left \{ [32]^1, [2]^{56}, [-8]^{18} \right \}$, and $\left \{ [40]^1, [2]^{75}, [-10]^{19} \right \}$.
	Such graphs have been shown not to exist by Bussemaker et al.~\cite{BHMW89:SRG49} and Azarija and Marc~\cite{Azarija:2015jk, Azarija:2016yu}.
	The nonexistence of Seidel matrices having either of the latter two spectra also preclude the existence of Seidel matrices having either of the spectra $\left \{ [-5]^{57}, [15]^{19} \right \}$ or $\left \{ [-5]^{76}, [19]^{20} \right \}$ \cite[Lemma 5.12]{Greaves2016208}.
\end{remark}

We have established a correspondence between Seidel matrices of odd order having precisely three distinct eigenvalues and strongly regular graphs on an odd number of vertices.
In fact, Lemma~\ref{lem:regSei} and Lemma~\ref{lem:regular} give a corresponding map for the eigenvalues of each object.
The eigenvalues of strongly regular graphs have been thoroughly investigated and, using the above mapping between the eigenvalues, we can transfer some of this study to the eigenvalues of Seidel matrices of odd order having three distinct eigenvalues.
In particular we can give the following characterisation of the spectrum of such matrices that possess at least one irrational eigenvalue.
However, we give a proof that is independent of results concerning strongly regular graphs.

\begin{corollary}\label{cor:irrational}
	Let $S$ be a Seidel matrix of order $n$ odd with precisely three distinct eigenvalues at least one of which is irrational.
	Then $S$ has spectrum
	\[
		\left \{ [-\sqrt{n}]^{(n-1)/2}, [0]^1, [\sqrt{n}]^{(n-1)/2} \right \}.
	\]
\end{corollary}
\begin{proof}
	By Corollary~\ref{cor:oddSeidel}, the matrix $S$ must have a simple eigenvalue ($\nu$, say).
	Furthermore, by Corollary~\ref{cor:regularParity}, we can assume that $\nu$ is an (even) integer.
	Let $\lambda$ be an irrational eigenvalue of $S$.
	The remaining eigenvalue of $S$ ($\mu$, say) must be the (Galois) conjugate of $\lambda$.
	Therefore we must have that $\lambda$ and $\mu$ are quadratic algebraic integers each having multiplicity $(n-1)/2$ as eigenvalues of $S$.
	
	Now, all quadratic algebraic integers can be written as $a + b\omega$ where $a$ and $b$ are rational integers and $\omega$ takes one of two possible forms:
	\[
		\omega = \begin{cases}
			\sqrt{d}, & \text{ for some positive integer } d \equiv 2, 3 \mod 4; \\
			\frac{1+\sqrt{d}}{2}, & \text{ for some positive integer } d \equiv 1 \mod 4.
		\end{cases}
	\]
	
	First assume that $\lambda, \mu = a \pm b\sqrt{d}$.
	Using \eqref{eqn:tr}, we can write
	\begin{equation}
		\label{eqn:trAPP}
		\frac{n-1}{2}(\lambda + \mu) + \nu = (n-1)a + \nu = 0.
	\end{equation}
	Putting \eqref{eqn:tr2} together with \eqref{eqn:trAPP} yields
	\begin{align*}
		  n(n-1) &= \frac{n-1}{2}(\lambda^2 + \mu^2) + \nu^2 \\
			&= (n-1)(a^2 + b^2 d) + (n-1)^2 a^2
	\end{align*}
	from which we find that $b^2 d = n(1-a^2)$.
	Since $b^2 d$ is positive ($\lambda$ and $\mu$ are distinct) and $a$ is an integer we must have that $a = 0$.
	The spectrum of $S$ is then forced to be $\{ [\sqrt{n}]^{(n-1)/2}, [0]^1, [-\sqrt{n}]^{(n-1)/2} \}$.
	
	Finally assume that $\lambda, \mu = a + b(1 \pm \sqrt{d})/2$ where $d \equiv 1 \mod 4$ and $b$ is odd.
	By \cite[Lemma 5.8]{Greaves2016208}, $n$ must be congruent to $1$ modulo $4$.
	Following the same reasoning as above we find that $b^2 d/4 = n (1 - (a+b/2)^2)$.
	Again, since $b^2 d/4$ is positive, we must have that $|a+b/2| < 1$ whose only solution is $|a+b/2| = 1/2$.
	Therefore we must have $b^2 d = 3n$.
	By reducing modulo $4$ we obtain a contradiction.
\end{proof}

\subsection{Seidel matrices of even order} 
\label{sec:seidel_matrices_of_even_order}

Let $S$ be a Seidel matrix of order $n$ even having spectrum $\{ [\lambda]^a, [\mu]^b, [\nu]^c \}$.
In this section we find sufficient conditions that guarantee that the switching class of $S$ contains a regular graph.
In Proposition~\ref{pro:regSei} we saw that $1 \in \{a,b,c\}$ suffices.
We have the following result for when $2 \in \{a,b,c\}$.

\begin{theorem}\label{thm:rank2}
	Let $S$ be a Seidel matrix of order $n$ even with spectrum $\{ [\lambda]^a, [\mu]^b, [\nu]^2 \}$.
	Assume that either 
		 $|n-1+\lambda \mu|$,
		 $|n-1+\lambda \mu|/2$, or
		 $|n-1+\lambda \mu|/4$,
	is a prime congruent to $3$ modulo $4$.
	Then the $\nu$-eigenspace of $S$ is regular.
\end{theorem}
\begin{proof}
	By Corollary~\ref{cor:Mform}, we have $M_S(\lambda,\mu) \cong M$ for some $M \equiv |n-1+\lambda \mu| J \mod 4\mathcal M_n$.
	Assume that $2^r$ divides $|n-1+\lambda \mu|$ for some $r \in \{ 0,1,2 \}$ and $p:= |n-1+\lambda \mu|/2^r  \equiv 3 \pmod 4$ is a prime.
	Define the matrix $N:= M/2^r$.
	Then $N$ is a positive semidefinite $\mathbb Z$-matrix with constant diagonal $p$ and rank $2$.
	By Lemma~\ref{lem:primeDiag}, we can assume that the matrix $N$ has precisely two distinct rows, and hence (since $N$ is symmetric) $N$ has precisely two distinct entries.
	It follows from Lemma~\ref{lem:diagM} that each row of $N$ has a constant row sum and hence the non-zero eigenspace of $N$ is regular (the $0$-eigenspace of $N$ is also regular).
	Therefore the $\nu$-eigenspace of $S$ is regular.
\end{proof}


\begin{remark}\label{rem:Seidelrank2}
Using Theorem~\ref{thm:rank2}, we find that there exists a Seidel matrix with spectrum $\{ [-5]^{30}, [7]^{16}, [19]^2 \}$ if and only if there exists a graph with spectrum $\{ [14]^1, [2]^{30}, [-4]^{16}, [-10]^1 \}$.
\end{remark}

Next we show that if the value $|n-1+\lambda\mu|$ is an integer less than $8$ then the eigenspace of $\nu$ is regular.
Since $n$ is even and $\lambda\mu$ is an integer, by Lemma~\ref{lem:Neumann}, we have that $\lambda\mu$ is odd.
Hence it suffices to consider $|n-1+\lambda\mu| \in \{0,2,4,6\}$.
Since $S$ has more than two eigenvalues $|n-1+\lambda\mu| = 0$ is impossible.
It therefore remains to consider $|n-1+\lambda\mu| \in \{2,4,6\}$, which we do in the next few results.

\begin{lemma}\label{lem:4diag}
	Let $S$ be a Seidel matrix of order $n$ even with spectrum $\{ [\lambda]^a, [\mu]^b, [\nu]^c \}$.
	Then the matrix $M_S(\lambda,\mu)$ is switching equivalent to 
	\begin{enumerate}[(a)]
		\item $2 J$ if $|n-1+\lambda \mu| = 2$;
		\item $4I_c \otimes J_{n/c}$ if $|n-1+\lambda \mu| = 4$.
	\end{enumerate}
\end{lemma}
\begin{proof}
	 By Corollary~\ref{cor:Mform}, we have $M_S(\lambda,\mu) \cong M$ for some $M \equiv |n-1+\lambda \mu| J \mod 4\mathcal M_n$.
	 The diagonal entries of $M$ are equal to $|n-1+\lambda \mu|$ and, since $M$ is positive semidefinite, each off-diagonal entry of $M$ has absolute value at most $|n-1+\lambda \mu|$.
	 In the case of (a) each entry of $M$ is equal to $\pm 2$ and for (b) each entry of $M$ is in $\{0,\pm 4\}$.
	 For case (b), apply Lemma~\ref{lem:diagM} to deduce that each row has the same number of entries equal to $\pm 4$.
	 The lemma then follows by applying Corollary~\ref{lem:repDiagTensor}.
\end{proof}

\begin{proposition}\label{pro:uniqueSeidels}
	Let $S$ be a Seidel matrix with one of the following spectra
	\[
		\left \{ [-3]^1, [1]^{3} \right \},\; \left \{ [-3]^5, [3]^{5} \right \},\; \left \{ [-3]^{10}, [5]^{6} \right \}, \text{ and } \left \{ [-3]^{21}, [9]^{7} \right \}.
	\]
	Then both the eigenspaces of $S$ are regular.
\end{proposition}
\begin{proof}
	The first spectrum is easy to check by hand.
	The rest follow from \cite[Theorem 3.7]{Taylor77}.
\end{proof}

Now we can prove the main lemma of this section.

\begin{lemma}\label{lem:6diag}
	Let $S$ be a Seidel matrix of order $n$ even with spectrum $\{ [\lambda]^a, [\mu]^b, [\nu]^c \}$ such that $|n-1+\lambda \mu| = 6$.
	Then the matrix $ M_S(\lambda,\mu)$ is switching equivalent to $2(T+3I) \otimes J_{n/q}$, for some Seidel matrix $T$ with spectrum $\{ [-3]^{q-c}, [(q-1)/3]^c  \}$ where $(q+8)(9-c) = 72$.
	Moreover, the $\nu$-eigenspace of $S$ is regular.
\end{lemma}
\begin{proof}
	By Corollary~\ref{cor:Mform}, we have $M_S(\lambda,\mu) \cong M$ for some $M \equiv |n-1+\lambda \mu| J \mod 4\mathcal M_n$.
	Hence $M \equiv 2J \mod 4 \mathcal M_n$ and, moreover, $M$ has constant diagonal entries, each equal to $6$.
	Using Lemma~\ref{lem:diagM}, we see that each row of $M$ has a constant number of entries whose absolute value is $6$.
	Now apply Corollary~\ref{lem:repDiagTensor} to deduce that $M$ is switching equivalent to $N \otimes J_{n/q}$ where $N$ is a symmetric matrix with each diagonal entry equal to $6$ and all other entries of absolute value strictly less than $6$.
	Since $ M$ has rank $c$ and precisely two distinct eigenvalues, the matrix $N$ must also have rank $c$ and precisely two distinct eigenvalues.
	Furthermore, observe that we can write $N = 2(T+3I)$ where $T$ is a Seidel matrix of order $q$ having spectrum $\{ [-3]^{q-c}, [\theta]^c  \}$ for some $\theta$.
	Using \eqref{eqn:tr} and \eqref{eqn:tr2}, we find that $\theta = (q-1)/3$ and $(q+8)(9-c) = 72$.
	
	Now, there are only five possibilities for the positive integers $q$ and $c$.
	Indeed, the tuple $(q,c)$ must belong to the set $R := \{ (1,1), (4,3), (10,5), (16,6), (28,7) \}$.
	When $q=1$, we have $2(T+3I) \otimes J_{n/q} = 6 J_n$ and clearly the all-ones vector is an eigenvector for the eigenvalue $6n$.
	Hence the $\nu$-eigenspace of $S$ is regular.
	For the tuples in $R$ with $q >1$, by applying Proposition~\ref{pro:uniqueSeidels}, we find that, in each case, the $\nu$-eigenspace of $S$ is regular.  
\end{proof}

\begin{remark}
	\label{rem:607290graph}
	We can apply Lemma~\ref{lem:6diag} to several putative Seidel matrices and a Seidel matrix already known to exist.
	\begin{itemize}
		\item There exists a Seidel matrix with spectrum $\{ [-5]^{42}, [11]^{15}, [15]^3 \}$ if and only if there exists a graph with spectrum $\{ [22]^1, [2]^{42}, [-6]^{15}, [-8]^2 \}$.
		\item There exists a Seidel matrix with spectrum $\{ [-5]^{53}, [13]^{16}, [19]^3 \}$ if and only if there exists a graph with spectrum $\{ [26]^1, [2]^{53}, [-7]^{16}, [-10]^2 \}$.
		\item There exists a Seidel matrix with spectrum $\{ [-5]^{70}, [13]^{5}, [19]^{15} \}$ if and only if there exists a graph with spectrum $\{ [38]^1, [2]^{70}, [-7]^{4}, [-10]^{15} \}$.
	\end{itemize}
\end{remark}

By examining the proof of Lemma~\ref{lem:6diag}, we can extract the following condition on the multiplicity of the eigenvalue $\nu$.

\begin{corollary}\label{cor:multRestriction}
	Let $S$ be a Seidel matrix of order $n$ even with spectrum $\{ [\lambda]^a, [\mu]^b, [\nu]^c \}$ such that $|n-1+\lambda \mu| = 6$.
	Then $c \in \{1,3,5,6,7\}$.
\end{corollary}

Corollary~\ref{cor:multRestriction} provides an alternative proof of the nonexistence of Seidel matrices matrices having the spectrum $\{ [-5]^{16}, [5]^{9}, [7]^{5} \}$ or the spectrum $\{ [-5]^{26}, [7]^{7}, [9]^{9} \}$ cf. \cite[Theorem 5.25]{Greaves2016208}.
Corollary~\ref{cor:multRestriction} also provides the first proof of the nonexistence of Seidel matrices having the spectrum $\{ [-5]^{24}, [5]^{6}, [9]^{10} \}$.
See Table~\ref{tab:largeSets}.

Putting together Lemma~\ref{lem:4diag} and Lemma~\ref{lem:6diag} gives us the main result of this section.

\begin{theorem}\label{thm:less8}
	Let $S$ be a Seidel matrix of order $n$ even with spectrum $\{ [\lambda]^a, [\mu]^b, [\nu]^c \}$.
	Suppose that $|n-1+\lambda \mu| $ is an integer less than $8$.
	Then the $\nu$-eigenspace of $S$ is regular.
\end{theorem}

\subsection{Description of Table~\ref{tab:largeSets}} 
\label{sub:a_table_of_spectra}

To conclude this section we give a list (Table~\ref{tab:largeSets}) of spectra for Seidel matrices that was generated in \cite[Table 5]{Greaves2016208}.
The columns titled `$n$' and `$d$' denote the cardinality and dimension of the associated set of equiangular lines.
We include a column (titled `Regular'), which indicates whether a Seidel matrix with the given spectrum contains a regular graph in its switching class.
We write `Y', `N', and `?' respectively if the switching class of a Seidel matrix having such a spectrum must, must not, or may contain a regular graph.

\begin{table}[htbp]
	\begin{center}
	\begin{tabular}{cclcccc}
		$n$ & $d$ & Spectrum & Regular & Exists & Remark\\
		\hline \vspace{-1.1em}\\
		 $28$ & $14$ & $\{ \left[-5\right]^{14},    \left[3\right]^7    ,  \left[7\right]^7      \}$  &  ?  & Y   & \cite[Example 5.17]{Greaves2016208} \\
		 $30$ & $14$ & $\{ \left[-5\right]^{16},    \left[5\right]^9    ,  \left[7\right]^5      \}$  &  Y  & N   & Lemma~\ref{lem:4diag}, Corollary~\ref{cor:multRestriction} \\
		 $40$ & $16$ & $\{ \left[-5\right]^{24},    \left[5\right]^6    ,  \left[9\right]^{10}   \}$  &  N  & N   & Remark~\ref{rem:noreg40}, Corollary~\ref{cor:multRestriction} \\
		 $40$ & $16$ & $\{ \left[-5\right]^{24},    \left[7\right]^{15} ,  \left[15\right]^1     \}$  &  Y  & Y   & Proposition~\ref{pro:regSei}, \cite[Example 5.10]{Greaves2016208} \\
		 $42$ & $16$ & $\{ \left[-5\right]^{26},    \left[7\right]^7    ,  \left[9\right]^9      \}$  &  Y  & N   & Lemma~\ref{lem:4diag}, Corollary~\ref{cor:multRestriction} \\
		 $48$ & $17$ & $\{ \left[-5\right]^{31},    \left[7\right]^8    ,  \left[11\right]^9     \}$  &  ?  & Y   & \cite[Example 5.18]{Greaves2016208} \\
		 $49$ & $17$ & $\{ \left[-5\right]^{32},    \left[9\right]^{16} ,  \left[16\right]^1     \}$  &  Y  & N   & Proposition~\ref{pro:regSei}, Remark~\ref{rem:no497595} \\ 
		 $48$ & $18$ & $\{ \left[-5\right]^{30},    \left[3\right]^6    ,  \left[11\right]^{12}  \}$  &  ?  & ?   & \\
		 $48$ & $18$ & $\{ \left[-5\right]^{30},    \left[7\right]^{16} ,  \left[19\right]^2     \}$  &  Y  & ?   & Remark~\ref{rem:Seidelrank2} \\
		 $54$ & $18$ & $\{ \left[-5\right]^{36},    \left[7\right]^9    ,  \left[13\right]^9     \}$  &  ?  & ?   & \\
		 $60$ & $18$ & $\{ \left[-5\right]^{42},    \left[11\right]^{15},  \left[15\right]^3     \}$  &  Y  & ?    & Remark~\ref{rem:607290graph} \\
		 $72$ & $19$ & $\{ \left[-5\right]^{53},    \left[13\right]^{16},  \left[19\right]^3     \}$  &  Y  & Y   &  Remark~\ref{rem:607290graph}, \cite[Example 5.19]{Greaves2016208}\\
		 $75$ & $19$ & $\{ \left[-5\right]^{56},    \left[10\right]^1   ,  \left[15\right]^{18}  \}$  &  Y  & N   & Proposition~\ref{pro:regSei}, Remark~\ref{rem:no497595} \\
		 $90$ & $20$ & $\{ \left[-5\right]^{70},    \left[13\right]^5   ,  \left[19\right]^{15}  \}$  &  Y  & ?   & Remark~\ref{rem:607290graph} \\
		 $95$ & $20$ & $\{ \left[-5\right]^{75},    \left[14\right]^1   ,  \left[19\right]^{19}  \}$  &  Y  & N   & Proposition~\ref{pro:regSei}, Remark~\ref{rem:no497595} \\
	\end{tabular}
	\end{center}
	\caption{Some Seidel matrices with precisely three distinct eigenvalues.}
	\label{tab:largeSets}
\end{table}



%
%
%

\section{Strengthening the relative bound} 
\label{sec:relative_bound}

In this section we explore further a strengthening of the relative bound that was given in \cite[Theorem 5.21]{Greaves2016208}.
We give an alternative presentation of the result and exhibit some of its consequences including further improvements to upper bounds on the size of equiangular lines in Euclidean space.
Particular consequences for equiangular line systems in $\mathbb R^d$ for $d \in \{14,\dots,23\}$ are given in Table~\ref{tab:extreme-5} below.

First we state the relative bound for Seidel matrices.

\begin{theorem}[Relative bound]\label{thm:relbnd}
	Let $d \geqslant 1$, let $S$ be a Seidel matrix of order $n \geqslant 2$ with smallest eigenvalue $\lambda_0$ of multiplicity $n - d \geqslant 1$, and suppose $\lambda_0^2 \geqslant d + 2$.
	Then
	\[
		n \leqslant \frac{d(\lambda_0^2-1)}{\lambda_0^2-d},
	\]
	with equality if and only if $S$ has spectrum $\{ [\lambda_0]^{n-d}, [-\lambda_0(n-d)/d]^{d} \}$.
\end{theorem}

Our next result is essentially just a restatement of \cite[Theorem 5.21]{Greaves2016208}, but for the sake of clarity we give a proof.

\begin{theorem}\label{thm:rel2}
	Let $d \geqslant 1$, let $S$ be a Seidel matrix of order $n \geqslant 2$ with smallest eigenvalue $\lambda_0 \in \mathbb Z$ of multiplicity $n - d \geqslant 1$, and let $\mu$ be an integer.
	Suppose $\mu \ne \lambda_0$, $\lambda_0^2 \geqslant d+2$, and $n = d(\lambda_0^2-1)/(\lambda_0^2-d) - t$ for some $t \geqslant 0$.
	Then
	\begin{equation}
		\label{eqn:inequ1}
		m:= \dim\ker(S-\mu I) \geqslant \frac{ t ( t(\lambda_0^2-d) - d(\lambda_0^2 - 1)) - (d\mu + \lambda_0(n-d))^2 + d^2}{d},
	\end{equation}
	with equality if and only if $S$ has spectrum $\{ [\lambda_0]^{n-d}, [\mu - 1]^w, [\mu]^m, [\mu+1]^{d-m-w} \}$, where $w = (d\mu + \lambda_0(n-d) + d - m)/2$.
\end{theorem}
\begin{proof}
	Let $\lambda_1 \leqslant \lambda_2 \leqslant \dots \leqslant \lambda_d$ be the eigenvalues of $S$ not equal to $\lambda_0$.
	On the one hand, using \eqref{eqn:tr} and \eqref{eqn:tr2}, we have
	\begin{align*}
		d\sum_{i=1}^d (\lambda_i-\mu)^2 &= d(n(n-1)-(n-d)\lambda_0^2+2\mu\lambda_0 (n-d) + d\mu^2) \\
			 								&= -t ( t(\lambda_0^2-d) - d(\lambda_0^2 - 1)) + (d\mu + \lambda_0(n-d))^2.
	\end{align*}
	On the other hand, by the arithmetic-geometric mean inequality we have
	\begin{align}
		\sum_{i=1}^d (\lambda_i-\mu)^2 \geqslant (d-m)\prod_{\substack{i=1 \\ \lambda_i \ne \mu}}^d (\lambda_i - \mu)^{2/d}. \label{ineq:AMGM}
	\end{align}
	Since the set $R = \{ \lambda_i - \mu : 1 \leqslant i \leqslant d, \lambda_i \ne \mu \}$ consists of nonzero algebraic integers closed under Galois conjugation, the product $\prod_{\xi \in R} \xi^2$ is a positive integer.
	Hence the right hand side of \eqref{ineq:AMGM} is at least $d - m$.
	
	In the case that $\sum_{i=1}^d (\lambda_i-\mu)^2 = d-m$, using \eqref{ineq:AMGM}, we see that $\lambda_i = \mu \pm 1$ for all $1 \leqslant i \leqslant d$ for which $\lambda_i \ne \mu$.
	Then, using \eqref{eqn:tr}, one can determine the value of $w$.
	Conversely, if $S$ has spectrum $\{ [\lambda_0]^{n-d}, [\mu - 1]^w, [\mu]^m, [\mu+1]^{d-m-w} \}$, where $w = (d\mu + \lambda_0(n-d) + d - m)/2$ then \eqref{ineq:AMGM} holds with equality and hence \eqref{eqn:inequ1} also holds with equality.
\end{proof}

\begin{remark}
	\label{rem:int}
	In the proof of Theorem~\ref{thm:rel2} we see that the right hand side of \eqref{eqn:inequ1} is an integer.  
\end{remark}


 Recall that Lemma~\ref{lem:Neumann} gives us a bound on the multiplicity of an even integer as an eigenvalue of a Seidel matrix.
Next we use this bound on the multiplicity together with Theorem~\ref{thm:rel2} to obtain the following two corollaries.

\begin{corollary}\label{cor:rel2Neu}
	Let $d \geqslant 1$, let $S$ be a Seidel matrix of order $n \geqslant 2$ with smallest eigenvalue $\lambda_0$ of multiplicity $n - d \geqslant 1$, and let $\mu$ be the closest even integer to $-\lambda_0(n-d)/d$.
	Suppose $\lambda_0^2 > d$, $n = \lfloor d(\lambda_0^2-1)/(\lambda_0^2-d) \rfloor$, and that
	\begin{equation}
		\label{eqn:inequ2}
		(n-d)(\lambda_0-\mu)^2 -n\mu^2 + d > n(n-1).
	\end{equation}
	Then $S$ has spectrum $\{ [\lambda_0]^{n-d}, [\mu - 1]^w, [\mu]^1, [\mu+1]^{d-1-w} \}$, where $w = (d\mu + \lambda_0(n-d) + d - 1)/2$.
\end{corollary}
\begin{proof}
	The inequality \eqref{eqn:inequ2} implies that the right hand side of inequality \eqref{eqn:inequ1} is strictly greater than $0$.
	Moreover, by Remark~\ref{rem:int}, this quantity must be at least $1$.
	On the other hand, since $\mu$ is an even integer, by Lemma~\ref{lem:Neumann}, the multiplicity of $\mu$ as an eigenvalue is at most $1$.
	Hence the left hand side of \eqref{eqn:inequ1} is at most $1$ and we have equality in \eqref{eqn:inequ1}.
	The result then follows from Theorem~\ref{thm:rel2}.
\end{proof}

\begin{corollary}\label{cor:rel2det}
	Let $d \geqslant 1$, let $S$ be a Seidel matrix of order $n \geqslant 2$ with smallest eigenvalue $\lambda_0$ of multiplicity $n - d \geqslant 1$, and let $\mu$ be the closest even integer to $-\lambda_0(n-d)/d$.
	Suppose $\lambda_0^2 > d$, $n = \lfloor d(\lambda_0^2-1)/(\lambda_0^2-d) \rfloor$ is even, and that
	\begin{equation}
		\label{eqn:inequ3}
		(n-d)(\lambda_0-\mu)^2 -n\mu^2 + d \geqslant n(n-1).
	\end{equation}
	Then $S$ has spectrum $\{ [\lambda_0]^{n-d}, [\mu - 1]^w, [\mu+1]^{d-w} \}$, where $w = (d\mu + \lambda_0(n-d) + d)/2$.
\end{corollary}

\begin{proof}
	The inequality \eqref{eqn:inequ3} implies that the right hand side of inequality \eqref{eqn:inequ1} is at least $0$.
	On the other hand, since $n$ and $\mu$ are both even integers, by Lemma~\ref{lem:Neumann}, $\mu$ cannot be an eigenvalue of $S$.
	Hence the left hand side of \eqref{eqn:inequ1} is at most $0$ and we have equality in \eqref{eqn:inequ1}.
	The result then follows from Theorem~\ref{thm:rel2}.
\end{proof}
%

Using Corollaries~\ref{cor:rel2Neu} and \ref{cor:rel2det} we list the spectra of the Seidel matrices that would correspond to equiangular line systems of order $n$ in $\mathbb R^d$ for $14 \leqslant d \leqslant 23$ where $n = \left \lfloor d(\lambda_0^2-1)/(\lambda_0^2-d) \right \rfloor$ and $\lambda_0 = -5$.

	\begin{table}[h]
	\begin{tabular}{l|l|c|c|c|c}
	 $d$ & $\lambda_0$ & $\left \lfloor \frac{d(\lambda_0^2-1)}{\lambda_0^2-d} \right \rfloor$ \vspace{0.05em} & Spectrum & Exists & Remark \\
	\hline
	 {$14$} & {$-5$} & {$30$} & {$\{ [-5]^{16}, [5]^9, [7]^5 \}$} & N & Corollary~\ref{cor:multRestriction}	\\
	 {$15$ } & {$-5$} & {$36$} & {$\{[-5]^{21}, [7]^{15} \}$ } & Y & See \cite{lemmens1973equiangular}	\\
	 {$16$} & {$-5$} & {$42$} & {$\{[-5]^{26}, [7]^7, [9]^9 \}$} & N & Corollary~\ref{cor:multRestriction} 	\\
	 {$17$} & {$-5$} & {$51$ }& {$\{[-5]^{34}, [10]^{17} \}$ } & N & Lemma~\ref{lem:Neumann} \\
	 $18$ & $-5$ & $61$ & $\{[-5]^{43}, [11]^9, [12]^1, [13]^8 \}$ & N & Remark~\ref{rem:61nonexi}	\\
	 $19$ & $-5$ & $76$ & $\{[-5]^{57}, [15]^{19} \}$ & N & Remark~\ref{rem:no497595}	\\
	 $20$ & $-5$ & $96$ & $\{[-5]^{76}, [19]^{20} \}$ & N & Remark~\ref{rem:no497595} \\
	 $21$ & $-5$ & $126$ & $\{[-5]^{105}, [25]^{21} \}$ & Y & See \cite{lemmens1973equiangular} \\
	 $22$ & $-5$ & $176$ & $\{[-5]^{154}, [35]^{22} \}$ & Y & See \cite{lemmens1973equiangular} \\
	 $23$ & $-5$ & $276$ & $\{[-5]^{125}, [55]^{21} \}$ & Y & See \cite{lemmens1973equiangular}
	\end{tabular}
	\caption{Some Seidel matrices corresponding to putative large sets of equiangular lines in dimensions $14$ to $23$.}
	\label{tab:extreme-5}
	\end{table}

%

\begin{remark}
	\label{rem:61}
	As shown in Table~\ref{tab:extreme-5}, by Corollary~\ref{cor:rel2Neu}, the existence of an equiangular line system of order $61$ in $\mathbb R^{18}$ is equivalent to the existence of a Seidel matrix having spectrum $\{[-5]^{43}, [11]^9, [12]^1, [13]^8 \}$.
	But, by Corollary~\ref{cor:posSemiTr}, (also see Remark~\ref{rem:61nonexi}) we see that a Seidel matrix having that spectrum cannot exist.
	Hence there does not exist an equiangular line system of order $61$ in $\mathbb R^{18}$.
\end{remark}


%


\section{Open Problems} 
\label{sec:open_problems}

We conclude this article with some open problems.

\begin{enumerate}[(i)]
	\item Which Seidel matrices having precisely three distinct eigenvalues have a regular graph in their switching class?
	
	See Question A and Question B in Section~\ref{sec:regular_graphs_in_the_switching_class}.
	
	\item Do all Seidel matrices having the spectrum $\left \{ [-5]^{14}, [3]^7, [7]^7 \right \}$ or the spectrum $\left \{ [-5]^{31}, [7]^8, [11]^9 \right \}$ have a regular graph in their switching classes?
	
	The known constructions \cite[Example 5.17 and Example 5.18]{Greaves2016208} both have a regular graph in their switching class.
	
	\item Does there exist a regular graph with spectrum $\{ [22]^1, [2]^{42}, [-6]^{15}, [-8]^2 \}$?
	
	The existence of such a graph corresponds to a system of $60$ equiangular lines in $\mathbb R^{18}$.
	See Remark~\ref{rem:607290graph}.
\end{enumerate}


\section*{Acknowledgements} 
\label{sec:acknowledgements}

The author is grateful to the referee for their careful reading of the manuscript.
The author has also benefitted from conversations with Jack Koolen, Akihiro Munemasa, and Ferenc Sz\"oll\H{o}si.


\bibliographystyle{plain}
\bibliography{sbib}

\end{document}